\documentclass[11pt]{amsart} 

\usepackage[utf8]{inputenc} 

\usepackage{geometry} 
\usepackage{graphicx} 

\usepackage{booktabs} 
\usepackage{array} 
\usepackage{paralist} 
\usepackage{verbatim} 
\usepackage{subfig} 

\usepackage{amsfonts}
\usepackage{amssymb}
\usepackage{amsmath}
\usepackage{verbatim}
\usepackage{hyperref}
\usepackage{color}

\newtheorem{theorem}{Theorem}    
\newtheorem{proposition}[theorem]{Proposition}

\newtheorem{lemma}[theorem]{Lemma}

\theoremstyle{definition}

\numberwithin{theorem}{section}
\numberwithin{definition}{section}
\numberwithin{equation}{section}

\usepackage{fancyhdr} 
\def\C{\mathbb{C}}

\def\R{\mathbb{R}}
\def\H{\mathbb{H}}

\newcommand{\ev}{\end{pmatrix}}

\def\bp{\mathbf{p}}
\def\bA{\mathbf{A}}
\def\bf{\mathbf{f}}
\def\bF{\mathbf{F}}
\def\bg{\mathbf{g}}
\def\bG{\mathbf{G}}
\def\G{\mathfrak{G}}
\def\bh{\mathbf{h}}

\def\T{\mathcal{T}}
\def\scriptO{\mathcal{O}}
\def\one{{\mathbf 1}}

\def\dist{\text{dist}}

\title{A Quantitative Stability Theorem for Convolution on the Heisenberg Group}
\author{Kevin O'Neill}

\address{
	Kevin O'Neill\\
        Mathematical Sciences Building\\
        One Shields Ave\\
        University of California \\
        Davis, CA 95616, USA}
\email{kwoneill@ucdavis.edu}

\begin{document}


\begin{abstract}
Although convolution on Euclidean space and the Heisenberg group satisfy the same $L^p$ bounds with the same optimal constants, the former has maximizers while the latter does not. However, as work of Christ has shown, it is still possible to characterize near-maximizers. Specifically, any near-maximizing triple of the trilinear form for convolution on the Heisenberg group must be close to a particular type of triple of ordered Gaussians after adjusting by symmetry. In this paper, we use the expansion method to prove a quantitative version of this characterization.
\end{abstract} 

\maketitle


\section{Introduction}

For triples of functions $\bf=(f_1,f_2,f_3)$ with $f_j:\R^d\to\C$, let

\begin{equation*}
\T_{\R^d}(\bf):=\iint_{\R^d\times\R^d}f_1(x)f_2(y)f_3(-x-y)dxdy
\end{equation*}
denote the trilinear form of convolution on $\R^d$.

In dual form, Young's convolution inequality states that for any triple of exponents $\bp=(p_1,p_2,p_3)\in[1,\infty]^3$ with $\sum_{j=1}^3p_j^{-1}=2$,

\begin{equation*}
|\T_{\R^d}(\bf)|\leq \prod_{j=1}^3\|f_j\|_{p_j}
\end{equation*}
for all $f_j\in L^{p_j}$ ($1\leq j\leq 3$). (In this scenario, we will write $\bf\in L^\bp$.) Such $\bp$ will be deemed \textit{admissible}.

Beckner \cite{MR0385456} and Brascamp and Lieb \cite{MR0412366} established the stronger statement that

\begin{equation}\label{eq: optimal Young's R^d}
|\T_{\R^d}(\bf)|\leq \bA_\bp^d \prod_{j=1}^3\|f_j\|_{p_j},
\end{equation}
where 
\begin{equation}\label{eq:optimal constant definition}
\bA_\bp=\prod_{j=1}^3p_j^{1/(2p_j)}/(p'_j)^{1/(2p'_j)}
\end{equation}
and $p'$ is the conjugate exponent to $p$. Moreover, $\bA_\bp^d$ is the optimal constant in \eqref{eq: optimal Young's R^d}.

Brascamp and Lieb \cite{MR0412366} showed that equality is attained in \eqref{eq: optimal Young's R^d} precisely when $\bf$ is the particular ordered triple of Gaussians $\bg=(g_1,g_2,g_3)$ with $g_j(z)=e^{-\pi p_j'|z|^2}:=e^{-\gamma_j|z|^2}$, or the orbit of $\bg$ under the symmetries of the operator: scaling, translation, modulation, and diagonal action of the general linear group $Gl(d)$. The non-trivial part of this result is the uniqueness of maximizers up to symmetry; the set of maximizers must be invariant under symmetries of the operator since they do not change the ratio of the left hand side of \eqref{eq: optimal Young's R^d} to the right hand side.

If equality is nearly attained in \eqref{eq: optimal Young's R^d} for a particular triple of functions $\bf$, then one would like to say that $\bf$ is close to $\bg$. As stated, this is false, as the symmetries of convolution may be used to send an $\bf$ which is close to $\bg$ to another near-maximizing triple, far from $\bg$. For example, if $f_j=g_j+\delta\phi$ for $\phi\in C_c^\infty(\R^d)$ and small $\delta>0$, then $\bf$ is close to $\bg$; yet, if $f_j'(x)=f_j(10^5x)$, $\T_{\R^d}(\bf')/\prod_j\|f_j'\|_{p_j}=\T_{\R^d}(\bf)/\prod_j\|f_j\|_{p_j}$ while $\bf'$ is far from $\bg$. (A formal definition of closeness will be given shortly.) However, one may obtain a positive result of this type, provided one merely conclude some element in the orbit of $\bf$ is close to $\bg$.

Let $\scriptO(\bf)$ denote the orbit of a triple of functions $\bf=(f_1,f_2,f_3)$ under the aforementioned symmetries. Define the distance function
\begin{equation}\label{eq:define distance R^d}
\dist_\bp(\scriptO(\bf),\bg):=\inf_{\bh\in\scriptO_C(\bf)}\max_j\|h_j-g_j\|_{p_j}.
\end{equation}

A recent result of Christ \cite{ChristSY} states:

\begin{theorem}\label{thm:ChristSY}
Let $K$ be a compact subset of $(1,2)^3$. For each $d\geq1$, there exists $c>0$ such that for all admissible $\bp\in K$ and all $\bf\in L^\bp(\R^d)$,
\begin{equation*}
|\T_{\R^d}(\bf)|\leq \left(\bA^d_\bp-c\dist_\bp(\scriptO(\bf),\bg)^2\right)\prod_j\|f_j\|_{p_j}.
\end{equation*}
\end{theorem}

A particular rephrasing of the conclusion of Theorem \ref{thm:ChristSY} states there exists $C>0$ such that if $|\T_{\R^d}(\bf)|\geq(1-\delta)\bA_\bp^d\prod_j\|f_j\|_{p_j}$, then $\dist_\bp(\scriptO_C(\bf),\bg)<C\sqrt{\delta}$. As originally stated in \cite{ChristSY}, the distance is between $\bf$ and the manifold of maximizing triples (i.e., $\scriptO(\bg)$), though the two definitions are easily shown to be equivalent after proper rescaling.

The purpose of this paper is to prove an analogue of Theorem \ref{thm:ChristSY} for the Heisenberg group.

The Heisenberg group $\mathbb{H}^d$ is the set $\R^{2d+1}$ identified as $\{z=(x,t):x\in \R^{2d},t\in\R\}$ with the group operation

\begin{equation*}
z\cdot z'=(x,t)\cdot(x',t'):=(x+x',t+t'+\sigma(x,x')),
\end{equation*}
where $\sigma$ is the symplectic form on $\R^{2d}$ defined by

\begin{equation*}
\sigma(x,x'):=\sum_{j=1}^d x_jx'_{j+d}-x_{j+d}x_j'.
\end{equation*}

In $\H^d$, the inverse of $(x,t)$ is $(-x,-t)$ and the Haar measure is Lebesgue measure on $\R^{2d+1}$.

Define the trilinear form for convolution on $\mathbb{H}^d$ by

\begin{equation*}
\T_{\H^d}(\bf)=\iint_{\H^d\times\H^d}f_1(z_1)f_2(z_2)f_3(z_2^{-1}z_1^{-1})dz_1dz_2.
\end{equation*}

It was shown by Klein and Russo \cite{KleinRusso} and Beckner \cite{Beckner2} that

\begin{equation}\label{eq: optimal Young's H^d}
|\T_{\H^d}(\bf)|\leq \bA_\bp^{2d+1} \prod_{j=1}^3\|f_j\|_{p_j},
\end{equation}
where $\bA_\bp^{2d+1}$ is optimal. (This is the same $\bA_\bp^{2d+1}$ as is defined in \eqref{eq:optimal constant definition}). Furthermore, Beckner observed that there are no maximizers of \eqref{eq: optimal Young's H^d}.

Consider the example $f_j(x,t)=e^{-\gamma_j(\lambda|x|^2+\lambda^{-1}t^2+i\lambda^{-1}t)}$, where $\lambda\to\infty$. Viewed as functions on $\R^{2d+1}$, $\bf$ is a maximizing triple for convolution for all $\lambda$. However, one may check by computation that $\bf$ is a maximizing sequence for convolution on $\H^d$ (that is, $\T_{\H^d}(\bf)/\prod_j\|f_j\|_{p_j}\to\bA_\bp^d$), yet equality is not attained in \eqref{eq: optimal Young's H^d} for any $\lambda$. Furthermore, the limit of $\bf$ as $\lambda\to\infty$ does not exist in $L^\bp(\H^d)$. While this is not a proof of Beckner's observation, it does provide a useful heuristic.

What accounts for this difference between $\R^{2d+1}$ and $\H^d$?  One explanation is that on $\R^{2d+1}$, the diagonal action of $Gl(2d+1)$ is a symmetry for convolution, allowing one to ``return" $\lambda|x|^2+\lambda^{-1}t^2$ to the $|x|^2+t^2$ found in the standard maximizing triple $\bg$. Furthermore, the modulation symmetry allows one to remove the oscillatory factor.

On $\H^d$ however, the symmetries of convolution do not include modulation in the $t$ variable nor the entirety of $Gl(2d+1)$. In some sense, these ``missing symmetries'' are the only obstacle to the existence of maximizers and all maximizing triples for convolution on $\H^d$ are close to a triple of similar form after adjusting by the appropriate symmetries (see the work of Christ, \cite{ChristHeisenberg}). The goal of this paper is to provide quantitative bounds for this closeness.

To state our main result requires a little more background.

Let $Sp(2d)$ denote the symplectic group on $\R^{2d}$, the set of matrices $S$ such that $\sigma(Sx,Sy)=\sigma(x,y)$ for all $x,y\in\R^{2d}$ with the group operation of matrix multiplication.

Formally, by a symmetry of $\T_{\H^d}$, we mean an operation on $\bf$ which preserves the ratio $|\T_{\H^d}(\bf)|/\prod_j\|f_j\|_{p_j}$. For $\H^d$, the symmetries of interest are:

\begin{itemize}
\item $f_j\mapsto a_jf_j$ for $a_j\in\C\setminus\{0\}$. (Scaling)
\item $f_j(x,t)\mapsto f_j(rx,r^2t)$ for $r\in(0,\infty)$. (Dilation)
\item $f_j(z)\mapsto f_j(u_jzw_j)$ with $w_1=u_2^{-1}, w_2=u_3^{-1}$, and $w_3=u_1^{-1}$. (Translation)
\item $f_j(x,t)\mapsto f_j(Sx,t)$ for $S\in Sp(2d)$. (Diagonal Action of the Symplectic Group)
\item $f_j(x,t)\mapsto f_j(x,t+\varphi(x))$, where $\varphi:\R^{2d}\to\R$ is linear. (Shear Transformation)
\item $f_j(x,t)\mapsto e^{iu\cdot x}f_j(x,t)$ for $u\in\R^{2d}$. (Modulation in $x$)
\end{itemize}

Let $\Sigma(\T_{\H^d})$ denote the group generated by the above symmetries. 
Following \cite{ChristHeisenberg}, we define a canonical $\epsilon$-diffuse Gaussian to be a function of the form $G(x,t)=e^{-|Lx|^2-at^2+ibt}$, where $a>0, b\in\R, L\in Gl(2d)$, and 
\begin{equation*}
\max\{a,a^{1/2},b\}\|L^{-1}\|<\epsilon.
\end{equation*}
Furthermore, given admissible $\bp$, a triple of canonical $\epsilon$-diffuse Gaussians of the form $G_j=e^{-|L_jx|^2-a_jt^2+ib_jt}$ is said to be $\bp$-admissible if there exist $L,a$, and $b$ such that $L_j=\gamma_j^{1/2}, a_j=\gamma_j a$, and $b_j=b$ for $1\leq j\leq 3$.

Lastly, we say an ordered triple $\mathbf{G}=(G_1,G_2,G_3)$ of Gaussians is $\epsilon$-diffuse and $\bp$-compatible if there exists $\Psi\in\Sigma(\T_{\H^d})$ and a $\bp$-compatible ordered triple $\tilde{\mathbf{G}}=(\tilde{G}_1,\tilde{G}_2,\tilde{G}_3)$ of canonical $\epsilon$-diffuse Gaussians such that $\mathbf{G}=\Psi\tilde{\mathbf{G}}$ for $1\leq j\leq 3$.

Our main result is the following:

\begin{theorem}\label{thm: main stability}
Let $d\geq1$ and $K\subset(1,2)^3$ be compact. Then, there exists a $C>0$ with the following property. Let $\bp\in K$ be admissible, $\bf\in L^\bp(\H^d)$, and $\|f_j\|_{p_j}\neq0$ for all $1\leq j\leq 3$. Let $\delta\in(0,1)$ and suppose that $|\T_{\H^d}(\bf)|\geq(1-\delta)\bA_\bp^{2d+1}\prod_j\|f_j\|_{p_j}$. Then there exists a $\bp$-compatible $C\sqrt{\delta}$-diffuse ordered triple of Gaussians $\bG=(G_1,G_2,G_3)$ such that
\begin{equation}\label{eq:main thm conclusion}
\|f_j-G_j\|_{p_j}<C\sqrt{\delta}\|f_j\|_{p_j}\text{    for }j\in\{1,2,3\}.
\end{equation}
\end{theorem}

The exponent 2 found in \eqref{eq:main thm conclusion} is sharp.

Prior work of Christ \cite{ChristHeisenberg} establishes a qualitative stability theorem of a similar form. This result is of the same form as Theorem \ref{thm: main stability}, yet refers to an undetermined function $\epsilon(\delta)$ satisfying $\lim_{\delta\to0}\epsilon(\delta)=0$ in place of $C\sqrt{\delta}$. We state this result as Theorem \ref{thm: Christ qualitative} and use it to reduce to small perturbations in Section \ref{sec: Reduce}.

Also in Section \ref{sec: Reduce}, we will develop a translation scheme between convolution on the Heisenberg group and a generalized operator which will allow us to prove Theorem \ref{thm: main stability} through the expansion method of Bianchi and Egnell \cite{BE}.

In Sections \ref{sec:Expansion in Group Structure} and \ref{sec:Expansion in Twisting Factor}, we compute some terms of the expansion. Section \ref{sec:convolution term} determines what is needed to apply a sharpened form of Young's inequality due to Christ to handle the remaining term of the expansion. In Section \ref{sec:Balancing Lemma}, we prove a balancing lemma to attain these conditions, allowing us to combine all the terms and conclude the proof of Theorem \ref{thm: main stability} in Section \ref{sec:main proof}.

In Section \ref{sec: variants}, we establish some variants of Theorem \ref{thm: main stability} in cases where at least one $p_j$ is greater than or equal to 2, as in \cite{ChristSY}. 

\textbf{Acknowledgment:} The author would like to thank Michael Christ for the suggestion of the problem and some helpful conversations.

\section{Translation Into a Distance}\label{sec: Reduce}

To prove Theorem \ref{thm: main stability}, we will use the expansion method. The first obstacle in performing the expansion is that, as previously discussed, there are no maximizers for convolution on the $\H^d$. Our solution is to use the maximizers for convolution on $\R^{2d+1}$ for comparison, continuously varying the group structure between that of the two spaces. This leads to another obstacle, in that the symmetry groups for convolution on $\H^d$ and $\R^{2d+1}$ differ. This is a problem because with the differing symmetry groups, a near-maximizer for convolution on $\H^d$ such as $f_j(x,t)=e^{-\gamma_j(\lambda|x|^2+\lambda^{-1}t^2+i\lambda^{-1}t)}$ cannot be sent to a small neighborhood of the fixed maximizing triple $g_j(z)=e^{-\gamma_j|z|^2}$ under the symmetries for $\H^d$.

To resolve this second issue, we introduce a new functional which generalizes the trilinear forms for convolution on both $\R^{2d+1}$ and $\H^d$. 
 This new functional effectively allows for a more flexible group structure, so general elements of $Gl(2d+1)$ may act as symmetries by changing this group structure. Similar ideas were applied to the case of twisted convolution by the author in \cite{MeTC}.

The generalized functional is:

\begin{equation*}
\T(\bf,A,b):=\iint f_1(z_1)f_2(z_2)f_3(-z_1-z_2-e_{2d+1}\sigma(Ax_1,Ax_2))e^{ib\sigma(Ax_1,Ax_2)}dz_1dz_2,
\end{equation*}
where $\bf\in L^{\bp}(\R^{2d+1})$, $A$ is a $(2d)\times(2d)$ matrix, $b\in\R$, and $e_{2d+1}=(0,...,0,1)\in\R^{2d+1}$. $(A,b)$ will be referred to as the \textit{attached parameters}.

Since one may view $\T$ as a convolution-like operator with varying group structure, it will be helpful to use $\cdot_A$ to 
denote the group operation $(x,t)\cdot_A(x',t'):=(x+x',t+t'+\sigma(Ax,Ax'))$.

Let $\tau^A_uf(z)=f(u\cdot_Az)$ and $\tilde{\tau}^A_wf(z)=f(z\cdot_Aw)$ to represent left and right translation, respectively. Modulation will be represented by the notation $M_\xi f(z)=e^{iz\cdot\xi}f(z)$. We will write $D_rf(x,t)=f(rx,r^2t)$ for dilation. Through slight abuse of notation, we will often write $f\circ A(x,t)=f(Ax,t)$ for linear maps $A:\R^{2d}\to\R^{2d}$ and $\bf\circ A$ to denote $(f_1\circ A,f_2\circ A,f_3\circ A)$.

By a symmetry of $\T$, we mean an operation on $(\bf,A,b)$ which preserves $\Phi(\bf,A,b):=|\T(\bf,A,b)|/\left(\prod_{j=1}^3\|f_j\|_{p_j}\right)$. Here, the relevant symmetries are:
\begin{itemize}
\item $\Phi((a_1f_1,a_2f_2,a_3f_3),A,b)=\Phi((f_1,f_2,f_3),A,b)$ for $a_j\in\C\setminus\{0\}$. (Scaling)
\item For $u_j=(U_j,U_j'),w_j=(W_j,W_j')\in\R^{2d}\times\R$ satisfying $w_1=u_2^{-1}, w_2=u_3^{-1}$, and $w_3=u_1^{-1}$, \begin{multline*}
\Phi(M_{-bA^TJA(U_2-U_3)}\tilde{\tau}_{w_1}^A\tau_{u_1}^Af_1,M_{-b(A(U_1-U_2))^TJA}\tilde{\tau}_{w_2}^A\tau_{u_2}^Af_2,\tilde{\tau}_{w_3}^A\tau_{u_3}^Af_3,A,b)\\
=\Phi(f_1,f_2,f_3,A,b).
\end{multline*}
(Translation-Modulation)
\item $\Phi(M_\xi f_1,M_\xi f_2,M_\xi f_3,A,b+\xi_{2d+1})=\Phi(f_1,f_2,f_3,A,b)$ for $\xi\in\R^{2d+1}$. (Modulation)
\item $\Phi(\bf\circ L,AL,b)=\Phi(\bf,A,b)$ for $A\in Gl(2d)$. (Diagonal Action of $Gl(2d)$)
\item $\Phi(D_rf_1,D_rf_2,D_rf_3,A,r^2b)=\Phi(f_1,f_2,f_3,A,b)$ for $r\in(0,\infty)$. (Dilation)
\item Let $\varphi:\R^{2d}\to\R$ be linear and let $g_j=f_j(x,t+\varphi(x))$. Then, $\Phi(\bg,A,b)=\Phi(\bf,A,b)$. (Sheer Transformation)
\end{itemize}
Note that 
$A$ changes precisely under $Gl(2d)$ symmetries, and $b$ changes precisely under modulations in the $(2d+1)$-st coordinate and dilation.

While some of the above symmetries may appear complicated, we will see shortly that for our purpose they may usually be applied in the special case $A=Id, b=0$, simplifying their expressions; here and throughout, $Id$ refers to the $(2d)\times(2d)$ identity matrix. (For instance, the translation-moldulation symmetry becomes
\begin{equation}\label{eq:simplified form}
\Phi(\tilde{\tau}_{w_1}^{Id}\tau_{u_1}^{Id}f_1,\tilde{\tau}_{w_2}^{Id}\tau_{u_2}^{Id}f_2,\tilde{\tau}_{w_3}^{Id}\tau_{u_3}^{Id}f_3,A,b)=\Phi(f_1,f_2,f_3,A,b),
\end{equation}
where $\tau^{Id}$ and $\tilde{\tau}^{Id}$ represent the usual translation on $\mathbb{H}^d$.) This is because the definition of orbit will allow for rather restricted use of symmetries.




At this point, one may expect to prove a direct analogue of Theorem \ref{thm:ChristSY} for the operator $\T(\bf,A,b)$. While this is possible, we desire something a little stronger to recover Theorem \ref{thm: main stability}. The $c\sqrt{\delta}$-diffuse Gaussians of Theorem \ref{thm: main stability} are not obtained through the action of \emph{any} symmetry of $\T(\bf,A,b)$ on $\bg$; rather, symmetries of $\T_{\H^d}$ are applied only \emph{after} symmetries of $\T(\bf,A,b)$ which are not symmetries of $\T(\H^d)$. Thus, we must define an alternative to the usual notion of orbit.


Let $\G_0$ be the group of symmetries of $\T(\bf,A,b)$ generated by $Gl(2d)$ and modulation in the $t$ variable and let $\G_1$ be the group generated by the remaining symmetries on the list, along with the diagonal action of $Sp(2d)$. Note that $\G_1$ is in one-to-one correspondence with $\Sigma(\T_{\H^d})$ when $A=Id$ and $b=0$. While $\Psi\in\G_j$ ($j=0,1$) is defined as acting on a tuple of the form $(\bf,A,b)$, we will often write $\Psi f_j$ to denote its action on a particular function. In such a scenario, we will attempt to be particularly clear on what $\Psi$ does to $f_j$, given the action may depend on the attached parameters.

Let $\tilde{\scriptO}(\bf,A,b)$ denote the set of elements of the form $\Psi_0\Psi_1((e^{-ibt}\bf)\circ A^{-1},Id,0)$ with $\Psi_j\in\G_j$ ($j=0,1$) and define

\begin{multline}\label{def:distance}
\dist_\bp(\tilde{\scriptO}(\bf,A,b),(\bg,0,0))^2:=\\\inf_{(\bh,M,r)\in\tilde{\scriptO}(\bf,A,b)}\max_j\|h_j-g_j\|_{p_j}^2+\|M^TJM\|^2+r^2\|M^TJM\|^2.
\end{multline}
Through this careful definition of distance, we will be able to recover Theorem \ref{thm: main stability} from the following result.

\begin{theorem}\label{thm: main distance}
Let $d\geq1$ and $K\subset(1,2)^3$ be a compact subset of admissible triples of exponents. Then, there exists $c>0$ such that for all $\bp\in K, \bf\in L^\bp(\R^{2d+1})$, $(2d)\times(2d)$ matrices $A$, and $b\in\R$,   
\begin{equation*}
|\T(\bf,A,b)|\leq(\bA_\bp^{2d+1}-c\dist_\bp(\tilde{\scriptO}(\bf,A,b),(\bg,0,0))^2)\prod_j\|f_j\|_{p_j}.
\end{equation*}
\end{theorem}

\begin{proof}[Proof of Theorem \ref{thm: main distance} $\Rightarrow$ Theorem \ref{thm: main stability}]
Suppose that $|\T_{\H^d}(\bf)|\geq(1-\delta)\bA_\bp^{2d+1}\prod_j\|f_j\|_{p_j}$. Since $\T_{\H^d}(\bf)=\T(\bf,Id,0)$, by Theorem \ref{thm: main distance}, we have
\begin{equation*}
\dist_\bp(\tilde{\scriptO}(\bf,Id,0),(\bg,0,0))<C\delta^{1/2}
\end{equation*}
for a constant $C$ independent of $\bf$.

By the definitions of $\tilde{\scriptO}(\bf,Id,0)$ and the distance function, there exist $\Psi\in \G_1$, $b\in\R$, and $A\in Gl(2d)$ such that

\begin{equation}\label{eq:by def of orbit}
\|e^{ibt}(\Psi f_j)\circ A-g_j\|=O(\delta^{1/2})\|g_j\|_{p_j},
\end{equation}
where $\|A^TJA\|=O(\delta^{1/2})$ and $b\|A^TJA\|=O(\delta^{1/2})$. Observe that since $\Psi$ acts on $(\bf,Id,0)$, one may view $\Psi$ as an element of $\Sigma(\T_{\H^d})$ (and we do so here in interpreting $\Psi f_j$). Noting that
\begin{equation}\label{eq: equal measurements}
\|A^TJA\|=\inf_{S\in Sp(2d)}\|S^{-1}\circ A\|^2,
\end{equation}
choose $S$ such that equality is attained and write $L=S^{-1}\circ A$.

By \eqref{eq:by def of orbit},

\begin{equation*}
\|f_j-\Psi^{-1}(e^{-ibt}g_j\circ L^{-1})\circ S^{-1}\|_{p_j}=O(\delta^{1/2})\|\Psi^{-1}(e^{-ibt}g_j\circ L^{-1})\circ S^{-1}\|_{p_j}=O(\delta^{1/2})\|f_j\|_{p_j}.
\end{equation*}
By \eqref{eq: equal measurements}, $\|L^{-1}\|^2=O(\delta^{1/2})$ and $b\|L^{-1}\|^2=O(\delta^{1/2})$, precisely the conclusion of Theorem \ref{thm: main stability}  taking $a=1$ in the definition of $\epsilon$-diffuse Gaussian.

\end{proof}

To prove Theorem \ref{thm: main distance}, we begin by reducing to small perturbations.

\begin{theorem}\label{thm: qualitative, distance version}
Let $d\geq 1$ and $K\subset(1,2)^3$ be a compact set of admissible ordered triples of exponents. Then, there exists a function $\delta\mapsto\epsilon(\delta)$ satisfying $\lim_{\delta\to0}\epsilon(\delta)=0$ with the following property. If $\bp\in K$, $\bf\in L^\bp(\R^{2d+1})$, $A$ is a $(2d)\times(2d)$ matrix, and $b\in\R$ such that $|\T(\bf,A,b)|\geq(1-\delta)\bA_\bp^{2d+1}\prod_j\|f_j\|_{p_j}$, then
\begin{equation*}
\dist_\bp(\tilde{\scriptO}(\bf,A,b),(\bg,0,0))<\epsilon(\delta).
\end{equation*}
\end{theorem}

Theorem \ref{thm: qualitative, distance version} is a qualitative stability result phrased in terms of $\T(\bf,A,b)$ and the corresponding distance function rather than convolution on $\H^d$. We will prove it by translating the following qualitative stability result of Christ \cite{ChristHeisenberg} into this scheme.

\begin{theorem}\label{thm: Christ qualitative}
Let $d\geq 1$ and $K\subset(1,2)^3$ be a compact set of admissible ordered triples of exponents. Then, there exists a function $\delta\mapsto\epsilon(\delta)$ satisfying $\lim_{\delta\to0}\epsilon(\delta)=0$ with the following property. Let $\bp\in K$, $\bf\in L^\bp(\H^d)$, and suppose that $\|f_j\|_{p_j}\neq0$ for each $j\in\{1,2,3\}$. Let $\delta\in(0,1)$ and suppose that $|\T_{\H^d}(\bf)|\geq(1-\delta)\bA_\bp^{2d+1}\prod_j\|f_j\|_{p_j}$. Then there exists a $\bp$-compatible $\epsilon(\delta)$-diffuse ordered triple of Gaussians $\bG=(G_1,G_2,G_3)$ such that
\begin{equation*}
\|f_j-G_j\|_{p_j}<\epsilon(\delta)\|f_j\|_{p_j}\text{    for }j\in\{1,2,3\}.
\end{equation*}
\end{theorem}

While the version of Theorem \ref{thm: Christ qualitative} stated in \cite{ChristHeisenberg} does not explicitly include the uniformity of $\epsilon(\delta)$ for $\bp\in K$, one may easily check this part of the conclusion is satisfied by reviewing the proof.

\begin{proof}[Proof of Theorem \ref{thm: Christ qualitative} $\Rightarrow$ Theorem \ref{thm: qualitative, distance version}]
By a standard approximation argument, it suffices to prove Theorem \ref{thm: qualitative, distance version} in the case of invertible matrices $A$, as each noninvertible matrix is arbitrarily close to an invertible matrix and $\T(\bf,A,b)$ is continuous in $A$. (Furthermore, in this scenario, one may choose the distance of the invertible matrix to depend on $\bf$.)

Suppose $\Phi(\bf,A,b)\geq (1-\delta)\bA_\bp^{2d+1}$. Then, applying the symmetries of $\T$, $\Phi(\bF,Id,0)\geq (1-\delta)\bA_\bp^{2d+1}$, where $\bF=e^{-ibt}\bf\circ A^{-1}$.

We now write $z=(x,t)$. By Theorem \ref{thm: Christ qualitative} (taking $a=1$ via the dilation symmetry), there exist $\Psi\in\Sigma(\T_{\H^d}$, $r\in\R$, and $L\in Gl(2d)$ such that
\begin{equation}\label{eq:functions close}
\frac{\|F_j-\Psi(e^{irt}g_j\circ L)\|_{p_j}}{\|F_j\|_{p_j}}<\epsilon(\delta)
\end{equation}
where 
$\max\{1,r\}\cdot\|L^{-1}\|^2<\epsilon(\delta)$. By \eqref{eq:functions close},
\begin{equation}\label{eq:close to gj}
\frac{\|e^{-irt}\Psi^{-1}F_j\circ L^{-1}-g_j\|_{p_j}}{\|e^{-irt}\Psi^{-1}F_j\circ L^{-1}\|_{p_j}}<\epsilon(\delta).
\end{equation}

Let $S_0$ denote the composition of symplectic matrices found in the symmetries which generate $\Psi$; that is, the matrix $S$ such that for any $(\bh,M,s)$, $\Psi(\bh,M,s)=(\tilde{\bh},SM,\tilde{s})$ for some $\tilde{\bh}\in L^\bp$ and $\tilde{s}\in\R$. Then, $(e^{-irt}\Psi^{-1}F_j\circ L^{-1},S_0^{-1}\circ L^{-1},-r)\in\tilde{\scriptO}(\bf,A,b)$. Since $\Psi$ acts in the case where the attached parameters are $Id$ and $0$, we may rightfully view it as an element of $\G_1$.

We now see that 
\begin{align*}
\dist_\bp(\tilde{\scriptO}(\bf,A,b),(\bg,0,0))^2\leq&\max_j\|e^{-irt}(\Psi^1)^{-1}F_j\circ L^{-1}-g_j\|_{p_j}^2\\&+(1+r^2)\inf_{S\in Sp(2d)}\|S\circ S_0^{-1}\circ L^{-1}\|^4\\
\leq& C\epsilon(\delta)^2+(1+r^2)\|L^{-1}\|^4\\
\leq& C\epsilon(\delta)^2.
\end{align*}

(In the above, we implicitly used the fact that \eqref{eq:close to gj} implies $\|e^{-irt}(\Psi^{-1}F_j)\circ L^{-1}\|_{p_j}$ is comparable to $\|g_j\|_{p_j}$.)

\end{proof}

\section{Expansion in Group Structure}\label{sec:Expansion in Group Structure}

By the translation scheme developed in Section \ref{sec: Reduce}, it suffices to prove Theorem \ref{thm: main distance}, and by Theorem \ref{thm: qualitative, distance version} it suffices to prove it under the assumption of small perturbations.

From here on, let $\bf\in L^\bp(\R^{2d+1})$ denote small, perturbative terms. For this section and the next, fix $b\in\R$ and a $(2d)\times(2d)$ matrix $A$. For $\bh\in L^\bp$, define $\|\bh\|_\bp:=\max_j\|h_j\|_{p_j}$. Our main object of interest is $\Phi(\bg+\bf,A,b)$. We write

\begin{multline}\label{eq:main expansion}
\frac{\T(\bg+\bf,A,b)}{\prod_j\|g_j+f_j\|_{p_j}}=\frac{\T(\bg+\bf,A,b)-\T(\bg+\bf,A,0)}{\prod_j\|g_j+f_j\|_{p_j}}\\+\frac{\T(\bg+\bf,A,0)-\T(\bg+\bf,0,0)}{\prod_j\|g_j+f_j\|_{p_j}}+\frac{\T(\bg+\bf,0,0)}{\prod_j\|g_j+f_j\|_{p_j}}
\end{multline}
and analyze each of the terms in the expansion in this and the following two sections. Control of the $\frac{\T(\bg+\bf,0,0)}{\prod_j\|g_j+f_j\|_{p_j}}$ term will follow partially from the analysis of \cite{ChristSY} and will be addressed in Section \ref{sec:Balancing Lemma}. Control of the other two terms on the right hand side will follow from a trilinear expansion in the function imputs and analysis similar to that of \cite{MeTC}.

Upon performing this expansion, the third-order terms will behave in a mildly unexpected manner. Specifically, they will be shown to not be $O(\|\bf\|_\bp^2\|A^TJA\|)$ through an example which involves the two $f_j$ moving out to infinity in opposite directions while minimizing the amount of cancellation. For this reason, it will be helpful to split the $f_j$ into pieces near to and far from the origin. The near terms will be analyzed immediately, while the far terms will be addressed later.

As in \cite{ChristHY} and \cite{ChristSY}, let $\eta>0$ be a small parameter to be chosen later (see Theorem \ref{thm:Christ intermediate}). For each $1\leq j\leq3$, decompose $f_j=f_{j,\sharp}+f_{j,\flat}$, where
\begin{equation}\label{eq:def fsharp}
f_{j,\sharp}=\left\{
     \begin{array}{lr}
       f_j(x) \text{ if }|f_j(x)|\leq \eta g_j(x)\\
       0 \text{ otherwise, }
     \end{array}
   \right.
\end{equation}
and $f_{j,\flat}=f_j-f_{j,\sharp}$.

The main result of this section is the following:

\begin{proposition}\label{prop:expand in group structure}
Let $d\geq 1$ and $\bp$ be an admissible triple of exponents. Then, there exists constant $C>0$ such that
\begin{multline*}
\frac{\T(\bg+\bf,A,0)-\T(\bg+\bf,0,0)}{\prod_j\|g_j+f_j\|_{p_j}}\le-C\|A^TJA\|^2\\+o((\|f\|_\bp+\|A^TJA\|)^2)+O(\|\bf_\sharp\|_\bp\|\bf_\flat\|_\bp+\|\bf_\flat\|_\bp^2).
\end{multline*}
\end{proposition}

The $o(\cdot)$ term will be deemed negligible by the reduction to small perturbations and the $O(\cdot)$ term will be counteracted by a negative term from our treatment of $\frac{\T(\bg+\bf,0,0)}{\prod_j\|g_j+f_j\|_{p_j}}$.

We begin by using the trilinearity of $\T$ to expand $\T(\bg+\bf,A,0)-\T(\bg+\bf,0,0)$ into 8 terms of the form

\begin{multline*}
T'(h_1,h_2,h_3):=\\
\iint h_1(z_1)h_2(z_2)\left[h_3(-z_1-z_2-e_{2d+1}\sigma(Ax_1,Ax_2))-h_3(-z_1-z_2)\right]dz_1dz_2,
\end{multline*}
where the $h_j$ are either all $g_j$, two $g_j$ and one $f_j$, one $g_j$ and two $f_j$, or all $f_j$. Since $\T(\cdot,A,0)$ may be written as the integral over the hypersurface $z_1\cdot_A z_2\cdot_A z_3=0$, the $h_j$ are interchangeable. For instance, bounds on $T'(f_1,g_2,g_3)$ immediately imply bounds on $T'(g_1,f_2,g_3)$ and $T'(g_1,g_2,f_3)$; similar implications hold for the case of one $g_j$ and two $f_j$.

The following two lemmas, proven in \cite{MeTC}, will be useful here. By minor abuse of notation, we let $g_j(w)=e^{-\gamma_j|w|^2}$ for $w\in\R, w\in\R^{2d}$, or $w\in\R^{2d+1}$.

\begin{lemma}\label{lemma:onesigma}
For all $f\in L^{p_1}(\R^{2d})$

\begin{equation*}
\iint_{\R^{2d}\times\R^{2d}} f(x)g_2(y)g_3(x+y)\sigma(Ax,Ay)dxdy=0
\end{equation*}
\end{lemma}

\begin{lemma}\label{lemma:twosigma}
For $\bg$ as above,
\begin{equation*}
\iint_{\R^{2d}\times\R^{2d}} g_1(x)g_2(y)g_3(x+y)\sigma^2(Ax,Ay)dxdy\approx_{d,\bp}\|A^TJA\|^2.
\end{equation*}
\end{lemma}

%
%

The following three lemmas will address the expansion in the current paper.

\begin{lemma}\label{lemma: T' 3g}
$T'(g_1,g_2,g_3)\leq -C_{d,\bp}\|A^TJA\|^2+O(\|A^TJA\|^3)$
\end{lemma}

\begin{proof}
By definition,
\begin{multline*}
T'(g_1,g_2,g_3)=\iint g_1(z_1)g_2(z_2)[g_3(-z_1-z_2-e_{2d+1}\sigma(Ax_1,Ax_2))\\
-g_3(-z_1-z_2)]dz_1dz_2\\
=\iint g_1(z_1)g_2(z_2)[e^{-\gamma_3((x_1+x_2)^2+(t_1+t_2+\sigma(Ax_1,Ax_2))^2}\\
-e^{-\gamma_3((x_1+x_2)^2+(t_1+t_2)^2)}]dz_1dz_2\\
=\iint g_1(z_1)g_2(z_2)g_3(-z_1-z_2)(e^{-\gamma_3(t_1+t_2)\sigma(Ax_1,Ax_2)-\gamma_3\sigma(Ax_1,Ax_2)^2}-1)dz_1dz_2.
\end{multline*}

By a simple Taylor expansion, the above is equal to

\begin{equation*}
\iint g_1(z_1)g_2(z_2)g_3(-z_1-z_2)\left[-2\gamma_3\alpha\beta-\gamma_3\beta^2+2\gamma_3^2\alpha^2\beta^2\right]dz_1dz_2+O(\|A^TJA\|^3),
\end{equation*}
where $\alpha:=(t_1+t_2)$ and $\beta:=\sigma(Ax_1,Ax_2)$. The higher powers of $\alpha$ and $\beta$ in the Taylor expansion lead to the $O(\|A^TJA\|^3)$ term because $|\sigma(Ax_1,Ax_2)|\leq \|A^TJA\|\cdot|x_1|\cdot|x_2|$ and the powers of $x$ and $t$ may be absorbed into the functions $g_j$. The resulting sum of integrals converges because the Taylor expansion for the exponential function has summable coefficients. (Formally, one may take the integral over the closed ball of radius $R$ in $\R^{2d}\times\R^{2d}$ so the Taylor expansion converges uniformly. Then, take the limit as $R\to\infty$. This reasoning will also be applied in later lemmas.)

By Lemma \ref{lemma:onesigma}, $\iint g_1(z_1)g_2(z_2)g_3(-z_1-z_2)\cdot2\gamma_3\alpha\beta dz_1dz_2=0$. 

By factoring the integral in $z_j$ into separate integrals over the $x_j$ and $t_j$, we see that
\begin{multline*}
\iint g_1(z_1)g_2(z_2)g_3(-z_1-z_2)\beta^2\gamma_3(2\gamma_3\alpha^2-1)dz_1dz_2\\
=\gamma_3\iint g_1(x_1)g_2(x_2)g_3(x_1+x_2)\sigma(Ax_1,Ax_2)^2dx_1dx_2\\
\times \iint g_1(t_1)g_2(t_2)g_3(t_1+t_2)(2\gamma_3(t_1+t_2)^2-1)dt_1dt_2.
\end{multline*}

Since $\iint g_1(x_1)g_2(x_2)g_3(x_1+x_2)\sigma(Ax_1,Ax_2)^2dx_1dx_2=C_{d,\bp}\|A^TJA\|^2$ by Lemma \ref{lemma:twosigma}, it suffices to show that

\begin{equation*}
\iint g_1(t_1)g_2(t_2)g_3(t_1+t_2)(2\gamma_3(t_1+t_2)^2-1)dt_1dt_2<0.
\end{equation*}

Completing the square, we see that

\begin{multline*}
\iint g_1(t_1)g_2(t_2)g_3(t_1+t_2)(2\gamma_3(t_1+t_2)^2-1)dt_1dt_2\\
=\iint e^{-(\gamma_1+\gamma_3)t_1^2-2\gamma_3t_1t_2}e^{-(\gamma_2+\gamma_3)t_2^2}(2\gamma_3(t_1+t_2)^2-1)dt_1dt_2\\
=\iint e^{-\left[(\sqrt{\gamma_1+\gamma_3)}t_1+\left(\frac{\gamma_3}{\sqrt{\gamma_1+\gamma_3}}t_2\right)\right]^2}e^{-\left[\gamma_2+\gamma_3-\frac{\gamma_3^2}{\gamma_1+\gamma_3}\right]t_2^2}(2\gamma_3(t_1+t_2)^2-1)dt_1dt_2.
\end{multline*}

We now make the change of variables $t_1\mapsto t_1-\frac{\gamma_3}{\gamma_1+\gamma_3}t_2, t_2\mapsto t_2$ so the above becomes

\begin{equation*}
\iint e^{-(\gamma_1+\gamma_3)t_1^2}e^{-\left[\gamma_2+\gamma_3-\frac{\gamma_3^2}{\gamma_1+\gamma_3}\right]t_2^2}(2\gamma_3(t_1+\left(1-\frac{\gamma_3}{\gamma_1+\gamma_3}\right)t_2)^2-1)dt_1dt_2.
\end{equation*}

Since the exponential terms are even in $t_1$ and $t_2$ and the cross terms $t_1t_2$ are odd in both variables, this is equal to

\begin{equation*}
\iint e^{-(\gamma_1+\gamma_3)t_1^2}e^{-\left[\gamma_2+\gamma_3-\frac{\gamma_3^2}{\gamma_1+\gamma_3}\right]t_2^2}(2\gamma_3t_1^2+2\gamma_3\left(\frac{\gamma_1}{\gamma_1+\gamma_3}\right)^2t_2^2-1)dt_1dt_2.
\end{equation*}

At this point, we use the fact that for $b>0$, $\int_\R r^2e^{-br^2}dr=\frac{1}{2b}\int_\R e^{-br^2}dr$. Letting $M=\iint e^{-(\gamma_1+\gamma_3)t_1^2}e^{-\left[\gamma_2+\gamma_3-\frac{\gamma_3^2}{\gamma_1+\gamma_3}\right]t_2^2}dt_1dt_2$, the integral in question is equal to

\begin{multline*}
M\left[\frac{2\gamma_3}{2(\gamma_1+\gamma_3)}+\frac{2\gamma_3}{2\left[\gamma_2+\gamma_3-\frac{\gamma_3^2}{\gamma_1+\gamma_3}\right]}\cdot\left(\frac{\gamma_1}{\gamma_1+\gamma_3}\right)^2-1\right]\\
=M\left[\frac{\gamma_3}{(\gamma_1+\gamma_3)}+\frac{\gamma_1^2\gamma_3}{(\gamma_1+\gamma_3)(\gamma_1\gamma_2+\gamma_1\gamma_3+\gamma_2\gamma_3)}-1\right]\\
=M\left[\gamma_3\frac{\gamma_1\gamma_2+\gamma_1\gamma_3+\gamma_2\gamma_3+\gamma_1^2}{(\gamma_1+\gamma_3)(\gamma_1\gamma_2+\gamma_1\gamma_3+\gamma_2\gamma_3)}-1\right]\\
=M\left[\gamma_3\frac{(\gamma_1+\gamma_3)(\gamma_1+\gamma_2)}{(\gamma_1+\gamma_3)(\gamma_1\gamma_2+\gamma_1\gamma_3+\gamma_2\gamma_3)}-1\right]\\
=M\left[\frac{\gamma_1\gamma_3+\gamma_2\gamma_3}{\gamma_1\gamma_2+\gamma_1\gamma_3+\gamma_2\gamma_3}-1\right]<0
\end{multline*}
since $M>0$ and $\gamma_j>0$ for all $j$.
\end{proof}


\begin{lemma}\label{lemma: T' 2g 1f}
$T'(f_1,g_2,g_3)=O(\|A^TJA\|^2\|f_1\|_{p_1})$.
\end{lemma}

\begin{proof}
Following the reasoning at the beginning of the proof of Lemma \ref{lemma: T' 3g}, we have
\begin{equation*}
T'(g_1,g_2,f_3)=\iint f_1(z_1)g_2(z_2)g_3(-z_1-z_2)[-2\gamma_3\alpha\beta]dz_1dz_2+O(\|A^TJA\|^2\|f_1\|_{p_1}),
\end{equation*}
where $\alpha=t_1+t_2$ and $\beta=\sigma(Ax_1,Ax_2)$ as before. The higher-order powers of $\alpha$ and $\beta$ provide a $O(\|A^TJA\|^2\|f_1\|{p_1})$ term since $|\sigma(Ax_1,Ax_2)|\leq \|A^TJA\|\cdot|x_1|\cdot|x_2|$ and powers may be absorbed into $g_2$ and $g_3$ to give an $L^{p_1'}$ function in $z_1$. As in the proof of Lemma \ref{lemma: T' 3g}, the integrals for all the powers are summable because the original Taylor expansion is summable.

The term coming from $\alpha\beta$ gives 0 by Lemma \ref{lemma:onesigma}.
\end{proof}

Naively, one may expect the term $T'(f_{1},f_{2},g_3)$ to be $O(\|A^TJA\|\cdot\|\bf\|_\bp^2)$. This is shown to be false by taking $\phi\in C_0^\infty(\R^{2d+1})$ to be a bump function near the origin, and letting $f_1(z_1)=\phi(z_1-(\lambda,...,\lambda)), f_2(z_2)=\phi(z_2+(\lambda,...,\lambda))$ as $|\lambda|\to\infty$. It is here that we will rely heavily on the properties of $f_{j,\sharp}$ in the decomposition $f_j=f_{j,\sharp}+f_{j,\flat}$.

\begin{lemma}\label{lemma: T' 1g 2f}
$T'(f_{1,\sharp},f_{2,\sharp},g_3)=o(\|\bf\|_{\bp}^2+\|A^TJA\|^2)$ with decay rate depending only on $\eta$.
\end{lemma}

In the proof of Lemma \ref{lemma: T' 1g 2f}, we will use the following trivial bound
\begin{equation}\label{eq:trivial bound}
|T'(h_1,h_2,g_3)|=O\left(\|h_1\|_{p_1}\|h_2\|_{p_2}\right)
\end{equation}
for arbitrary functions $h_j\in L^{p_j}$. The proof mimics that of Lemma 3.2 in \cite{MeTC}.

\begin{proof}
First, suppose that $\|A^TJA\|^3\geq\|f_{1,\sharp}\|_{p_1}\|f_{2,\sharp}\|_{p_2}$. Note that $\|A^TJA\|$ may be taken small enough that $\|A^TJA\|^3\leq\|A^TJA\|^2$ by our reduction to small perturbations in Theorem \ref{thm: qualitative, distance version}. By \eqref{eq:trivial bound},
\begin{equation*}
T'(f_{1,\sharp},f_{2,\sharp},g_3)\leq C\|f_{1,\sharp}\|_{p_1}\|f_{2,\sharp}\|_{p_2}\leq\|A^TJA\|^3=o(\|\bf\|_\bp^2+\|A^TJA\|^2)
\end{equation*}
and we are done.

So suppose that $\|A^TJA\|^3<\|f_{1,\sharp}\|_{p_1}\|f_{2,\sharp}\|_{p_2}$. Now, for $j=1,2$, write $f_{j,\sharp}=f_{j,\sharp,\leq M_j}+f_{j,\sharp,>M_j}$, where $f_{j,\sharp,\leq M_j}=f_{j,\sharp}\one_{B(0,M_j)}$ and $f_{j,\sharp,>M_j}=f_{j,\sharp}\one_{B(0,M_j)^c}$. Here, $\one_E$ refers to the indicator function of the set $E$, $B(z_0,R)$ refers to the closed ball of radius $R$ centered at $z_0$, $E^c$ is the complement of the set $E$, and $M_j$ is chosen so that
\begin{equation}\label{eq:smallnorm}
\|f_{j,\sharp,>M_j}\|_{p_j}=\|f_{j,\sharp}\|_{p_j}^2.
\end{equation}

Note that $M_j$ is dependent on $\eta$.

An elementary calculation shows $M_j\leq C\log(\|f_{j,\sharp}\|_{p_j}^{-1})$. (See the proof of Lemma 3.2 in \cite{MeTC} for details.)

Expand 
\begin{multline*}
T'(f_{1,\sharp},f_{2,\sharp},g_3)=T'(f_{1,\sharp,>M_1},f_{2,\sharp,>M_2},g_3)+T'(f_{1,\sharp,>M_1},f_{2,\sharp,\leq M_2},g_3)\\
+T'(f_{1,\sharp,\leq M_1},f_{2,\sharp,>M_2},g_3)+T'(f_{1,\sharp,\leq M_1},f_{2,\sharp,\leq M_2},g_3)
\end{multline*}
The first three of these terms are shown to be $O(\|\bf\|_\bp^3)$ by combining the trivial bound \eqref{eq:trivial bound} with \eqref{eq:smallnorm}.

Let $R=B(0,M_1)\times B(0,M_2)\subset \R^{2d}\times\R^{2d}$. Recall by our earlier Taylor expansion that

\begin{multline*}
T'(f_{1,\sharp,\leq M_1},f_{2,\sharp,\leq M_2},g_3)=\iint_R f_{1,\sharp,\leq M_1}(z_1)f_{2,\sharp,\leq M_2}(z_2)g_3(-z_1-z_2)\\\times\left[-2\gamma_3\alpha\beta-\gamma_3\beta^2+2\gamma_3^2\alpha^2\beta^2+O(\alpha^3\beta^3)+O(\beta^4)\right]dz_1dz_2.
\end{multline*}
In this case, the justification for inclusion of $O(\cdot)$ terms in the integrand is that the integral is over a compact domain; thus, the Taylor expansion converges uniformly.

We see that the absolute value of the integral term containing $-2\gamma_3\alpha\beta$ may be controlled by 

\begin{multline*}
C\iint_R|f_{1,\sharp}(z_1)|\cdot|f_{2,\sharp}(z_2)|\cdot g_3(-z_1-z_2)\|A^TJA\|\cdot|x_1|\cdot|x_2|\cdot|t_1+t_2|dz_1dz_2\\
\leq C\|f_{1,\sharp}\|_{p_1}\|f_{2,\sharp}\|_{p_2}\|A^TJA\|M_1^{3/2}M_2^{3/2}\\
\leq C\|f_{1,\sharp}\|_{p_1}^{4/3}\|f_{2,\sharp}\|_{p_2}^{4/3}\log(\|f_{1,\sharp}\|_{p_1})^{-3/2}\log(\|f_{2,\sharp}\|_{p_2})^{-3/2}=o(\|\bf\|_\bp^2)
\end{multline*}

The remaining terms may be dealt with similarly, the only difference being that different powers of $\|A^TJA\|$ and $M_j$ are obtained; however, the end result is always $o(\|\bf\|_\bp^2)$.

\end{proof}

\begin{proof}[Proof of Proposition \ref{prop:expand in group structure}]
Begin by using the trilinearity of $\T$ to expand $\T(\bg+\bf,A,0)-\T(\bg+\bf,0,0)$, expanding again via the decomposition $f_j=f_{j,\sharp}+f_{j,\flat}$ when terms contain two $f_j$ and one $g_j$. Applying Lemmas \ref{lemma: T' 3g}, \ref{lemma: T' 2g 1f}, and \ref{lemma: T' 1g 2f} to the resulting terms, considering that they apply equally after permutation of indices. The term with three $f_j$'s is trivially $O(\|\bf\|_\bp^3)$.

Note that we may ignore the division by $\prod_j\|g_j+f_j\|_{p_j}$ since for small $\|\bf\|_\bp$, this term is approximately the constant value $\prod_j\|g_j\|_{p_j}$; this only results in minor modifications to the constants in the right hand side of the conclusion.

The remaining terms are of the form $T'(f_{1,\sharp},f_{2,\flat},g_3), T'(f_{1,\flat},f_{2,\flat},g_3)$, or any of the similar forms obtained by permutations; hence, they may not be addressed by Lemma \ref{lemma: T' 1g 2f}. However, they may still be controlled by the trivial bound \eqref{eq:trivial bound}, resulting in the $O(\|\bf_\sharp\|_\bp\|\bf_\flat\|_\bp+\|\bf_\flat\|_\bp^2)$ term.
\end{proof}

%

\section{Expansion in Twisting Factor}\label{sec:Expansion in Twisting Factor}

Fix $b\in\R$ and $(2d)\times(2d)$ matrix $A$. Define

\begin{equation}\label{eq: def T''}
T''(\bh):=\iint h_1(z_1)h_2(z_2)h_3(-z_1-z_2-e_{2d+1}\sigma(Ax_1,Ax_2))\left[e^{ib\sigma(Ax_1,Ax_2)}-1\right]dz_1dz_2
\end{equation}

We analyze the expansion of the difference term

\begin{equation*}
T''(\bg+\bf)=\T(\bg+\bf,A,b)-\T(\bg+\bf,A,0).
\end{equation*}

The main result of this section is the following:

\begin{proposition}\label{prop:expansion in twisting}
Let $d\geq1$ and $\bp$ be an admissible triple of exponents. Then, there exists $C>0$ such that
\begin{multline*}
\frac{\T(\bg+\bf,A,b)-\T(\bg+\bf,A,0)}{\prod_j\|g_j+f_j\|_{p_j}}\le-Cb^2\|A^TJA\|^2+o([(1+b^2)^{1/2}\|A^TJA\|+\|\bf\|_\bp]^2)\\
+O(\|\bf_\sharp\|_\bp\|\bf_\flat\|_\bp+\|\bf_\flat\|_\bp^2)
\end{multline*}
\end{proposition}

As in the previous section, the trilinearity of $T''$ gives us 8 terms, each of which has three $g_j$, two $g_j$ and one $f_j$, one $g_j$ and two $f_j$, or 3 $f_j$.

\begin{lemma}\label{lemma: T'' 3g}
$T''(g_1,g_2,g_3)=-C_{d,\bp}b^2\|A^TJA\|^2+O(\|A^TJA\|^3+b^3\|A^TJA\|^3)$.
\end{lemma}

\begin{proof}
As in the proof of Lemma \ref{lemma: T' 3g}, we use a Taylor expansion, obtaining

\begin{equation}\label{eq: Taylor g3}
g_3(-z_1-z_2-e_{2d+1}\sigma(Ax_1,Ax_2))=g_3(-z_1-z_2)[1-2\gamma_3\alpha\beta+O(\beta^2)].
\end{equation}

Here, and again when powers of $\beta$ and $\alpha$ are used with $O(\cdot)$ notation, by $O(\beta^2)$ we mean that the remaining powers of $\beta$ in the Taylor expansion are of degree 2 or higher. (Issues of convergence may be addressed as in Section \ref{sec:Expansion in Group Structure}.)

Similarly,
\begin{equation}\label{eq: Taylor b term}
1-e^{ib\sigma(Ax_1,Ax_2)}=ib\beta-b^2\beta^2+O(b^3\beta^3).
\end{equation}

We now plug the product of \eqref{eq: Taylor g3} and \eqref{eq: Taylor b term} into \eqref{eq: def T''} with $h_j=g_j$. Each product of a $O(\cdot)$ term with another term gives a new term which is $O(\beta^3+b^3\beta^3)$ and the resulting integral is $O(\|A^TJA\|^3+b^3\|A^TJA\|^3)$.

Factoring out the $g_3(-z_1-z_2)$, the remaining terms are $ib\beta-b^2\beta^2-2ib\gamma_3\alpha\beta^2-2b^2\gamma_3\alpha\beta^3$. We first note that any integral with a single power of $\alpha$ must give 0 since
\begin{equation*}
\iint e^{-\gamma_1t_1^2}e^{-\gamma_2t_2^2}e^{-\gamma_3(t_1+t_2)^2}(t_1+t_2)dt_1dt_2=0
\end{equation*}
and the powers of $\beta$ only effect the integral in $x_1,x_2$. Second, the $ib\beta$ term gives 0 by Lemma \ref{lemma:onesigma}. The remaining term gives the integral
\begin{equation*}
\iint g_1(z_1)g_2(z_2)g_3(-z_1-z_2)b^2\sigma(Ax_1,Ax_2)^2dz_1dz_2=C_{d,\bp}b^2\|A^TJA\|^2,
\end{equation*}
by Lemma \ref{lemma:twosigma}.
\end{proof}

\begin{lemma}\label{lemma: T'' 2g 1f}
$T''(f_1,g_2,g_3)=O(\|f_1\|_{p_1}(\|A^TJA\|^2+b^2\|A^TJA\|^2)).$
\end{lemma}

\begin{proof}
By the Taylor expansions given in the proof of Lemma \ref{lemma: T'' 3g}, it suffices to determine bounds for

\begin{equation*}
\iint f_1(z_1)g_2(z_2)g_3(-z_1-z_2)\left[ib\beta-b^2\beta^2-2ib\gamma_3\alpha\beta^2-2b^2\gamma_3\alpha\beta^3+O(\beta^3+b^3\beta^3)\right]dz_1dz_2.
\end{equation*}

As before, the $ib\beta$ term vanishes by Lemma \ref{lemma:onesigma}. For the $-b^2\beta^2$ term, we see that

\begin{multline*}
\left|\int f_1(z_1)\left[ \int(-b^2\beta^2)g_2(z_2)g_3(-z_1-z_2)dz_2\right]dz_1\right|\\ \leq b^2\|A^TJA\|^2\int |f_1(z_1)|\left[\int x_1^2x_2^2g_2(z_2)g_3(-z_1-z_2)dz_2\right]dz_1
\end{multline*}
is $O(\|f_1\|_{p_1}(b^2\|A^TJA\|^2))$ since $\int x_1^2x_2^2g_2(z_2)g_3(-z_1-z_2)dz_2\in L^{p_1'}(z_1)$.

The remaining terms may be dealt with similarly

\end{proof}

\begin{lemma}\label{lemma: T'' 1g 2f}
$T''(f_{1,\sharp},f_{2,\sharp},g_3)=o(\|\bf\|_\bp^2+(1+b^2)\|A^TJA\|^2)$.
\end{lemma}

\begin{proof}
By the reduction to small perturbations, one may take $\|A^TJA\|\leq 1$ so that $\|A^TJA\|^3\leq \|A^TJA\|^2$. Also take $b\|A^TJA\|\leq 1$ so $b^3\|A^TJA\|^3\leq b^2\|A^TJA\|^2$.

As in the proof of Lemma \ref{lemma: T' 1g 2f}, the case of $\|A^TJA\|^3\geq \|f_{1,\sharp}\|_{p_1}\|f_{2,\sharp}\|_{p_2}$ is taken care of by the trivial bound
\begin{equation*}
|T''(h_1,h_2,g_3)|=O\left(\|h_1\|_{p_1}\|h_2\|_{p_2}\right)
\end{equation*}
The case of $b^3\|A^TJA\|^3\geq \|f_{1,\sharp}\|_{p_1}\|f_{2,\sharp}\|_{p_2}$ may be dealt with similarly.

So suppose $\|A^TJA\|^3<\|f_{1,\sharp}\|_{p_1}\|f_{2,\sharp}\|_{p_2}$ and $b^3\|A^TJA\|^3<\|f_{1,\sharp}\|_{p_1}\|f_{2,\sharp}\|_{p_2}$. Let $M_j$ and $R$ be as in the proof of Lemma \ref{lemma: T' 1g 2f}. Thus, by the proof of Lemma \ref{lemma: T' 1g 2f}, it suffices to bound $|T''(f_{1,\sharp,\leq M_1},f_{2,\sharp,\leq M_2},g_3)|$.

By Taylor expansion,

\begin{multline*}
T''(f_{1,\sharp,\leq M_1},f_{2,\sharp,\leq M_2},g_3)=\iint_R f_{1,\sharp,\leq M_1}(z_1)f_{2,\sharp,\leq M_2}(z_2)g_3(-z_1-z_2)\\\times(ib\beta-b^2\beta^2-2ib\gamma_3\alpha\beta^2-2b^2\gamma_3\alpha\beta^3+O(\beta^3+b^3\beta^3))dz_1dz_2.
\end{multline*}

Since $|\beta|\leq \|A^TJA\|\cdot|x_1|\cdot|x_2|$, the integral term coming from $ib\beta$ is controlled by 

\begin{align*}
\|f_{1,\sharp}\|_{p_1}\|f_{2,\sharp}\|_{p_2}b\|A^TJA\|M_1M_2
&C\leq \|f_{1,\sharp}\|_{p_1}^{4/3}\|f_{2,\sharp}\|_{p_2}^{4/3}\log(\|f_{1,\sharp}\|_{p_1}^{-1})\log(\|f_{2,\sharp}\|_{p_2}^{-1})\\&=o(\|\bf\|_\bp^2)
\end{align*}

The remaining terms may be dealt with similarly, instead obtaining different powers of $\|A^TJA\|, M_1$, and $M_2$, though in each case, one may check that the final result is $o(\|\bf\|_\bp^2)$ due to the presence of a power of $\|A^TJA\|$ and the $\log$ bounds for $M_j$.
\end{proof}

\begin{proof}[Proof of Proposition \ref{prop:expansion in twisting}]
As in the proof of Proposition \ref{prop:expand in group structure}, the result nearly follows from the trilinearity of $T''$, this time combined with the results of Lemmas \ref{lemma: T'' 3g}, \ref{lemma: T'' 2g 1f}, and \ref{lemma: T'' 1g 2f}. Again, there are terms with $f_{j,\flat}$ terms remaining, though by the trivial bound they result in the $O(\|\bf_\sharp\|_\bp\|\bf_\flat\|_\bp+\|\bf_\flat\|_\bp^2)$ term.
\end{proof}


\section{Treating the Euclidean Convolution Term}\label{sec:convolution term}


One may hope to complete the proof of Theorem \ref{thm: main stability} by applying Propositions \ref{prop:expand in group structure} and \ref{prop:expansion in twisting} to the expansion found in \eqref{eq:main expansion}, along with Theorem \ref{thm:ChristSY} to address the $\frac{\T(\bg+\bf,0,0)}{\prod_j\|g_j+f_j\|_{p_j}}$ term. However, Theorem \ref{thm:ChristSY} applies when $\bf$ represents the projective distance in \eqref{eq:define distance R^d}, and in our case, the projective distance \eqref{def:distance} is used-- which might not be comparable.

Rather than repeat the entire analysis of \cite{ChristSY}, it suffices to extract an intermediate theorem proven implicitly in the paper. To state this theorem requires some more definitions.

For $t>0$ and $n=0,1,2,...,$ let $P_n^{(t)}$ denote the real-valued polynomial of degree $n$ with positive leading coefficient and $\|P_n^{(t)}e^{-t\pi x^2}\|_{L^2(\R)}=1$ which is orthogonal to $P_k^{(t)}e^{-t\pi x^2}$ for all $0\leq k<n$.

For $d>1$, $\alpha=(\alpha_1,...,\alpha_{2d+1})\in\{0,1,2,...\}^{2d+1}$, and $x=(x_1,...,x_{2d+1})\in\R^{2d+1}$, define
\begin{equation*}
P_\alpha^{(t)}(x)=\prod_{k=1}^{2d+1}P_{\alpha_k}^{(t)}(x_k).
\end{equation*}

Lastly, for $1\leq j\leq 3$, let $\tau_j=\frac{1}{2}p_jp_j'$.

\begin{theorem}[\cite{ChristSY}]\label{thm:Christ intermediate}
Let $\delta_0>0$ be sufficiently small. There exists $c,\tilde{c}>0$ and a choice of $\eta>0$ in the $f_j=f_{j,\sharp}+f_{j,\flat}$ decomposition such that the following holds. Suppose $\|\bf\|_\bp<\delta_0$ and $f_j$ satisfy the following orthogonality conditions:
\begin{itemize}
\item $\langle Re(f_j),P_\alpha^{(\tau_j)}g_j^{p_j-1}\rangle=0$ whenever $\alpha=0$, $|\alpha|=1$ and $j\in\{1,2\}$, or $|\alpha|=2$ and $j=3$.
\item $\langle Im(f_j),P_\alpha^{(\tau_j)}g_j^{p_j-1}\rangle=0$ whenever $\alpha=0$ or $|\alpha|=1$ and $j=3$.
\end{itemize}
Then,
\begin{equation}\label{eq:penultimate}
\frac{\T_0(\bg+\bf)}{\prod_j\|g_j+f_j\|_{p_j}}\leq \bA_\bp^{2d}-c\|\bf\|_\bp^2-\tilde{c}\sum_j\|f_{j,\flat}\|_{p_j}^{p_j}.
\end{equation}
\end{theorem}
While Theorem \ref{thm:Christ intermediate} was not stated explicitly in \cite{ChristSY}, \eqref{eq:penultimate} is effectively the penultimate line in the proof of Theorem \ref{thm:ChristSY}, which used the orthogonality conditions in clear fashion.

Note that the $-\tilde{c}\sum_j\|f_{j,\flat}\|_{p_j}^{p_j}$ will be useful in canceling out the contribution of the $O(\|\bf^\sharp\|_\bp\|\bf^\flat\|_\bp+\|\bf^\flat\|_\bp^2)$ term.

The following section will show that one may reduce to the case in which the orthogonality conditions of Theorem \ref{thm:Christ intermediate} hold.

\section{Balancing Lemma}\label{sec:Balancing Lemma}

In this section, we prove a Balancing Lemma which will allow us to replace a given $(\bf,A,b)$ with a nearby one in its orbit that satisfies the orthogonality conditions of Theorem \ref{thm:Christ intermediate}.

\begin{lemma}[Balancing Lemma]\label{lemma: balancing}
Let $d\geq 1$ and $\bp\in(1,2]^3$ with $\sum_jp_j^{-1}=2$. There exists $\delta_0>0$ such that if
\begin{equation}\label{eq:distance small}
\dist_\bp(\tilde{\scriptO}(\bF,A,b),(\bg,0,0))<\delta_0,
\end{equation}
then there exists $(\tilde{\bF},\tilde{A},\tilde{b})\in\tilde{\scriptO}(\bF,A,b)$ such that the orthogonality conditions of Theorem \ref{thm:Christ intermediate} are satisfied for $\tilde{\bF}$. 
\end{lemma}

\begin{proof}
Suppose \eqref{eq:distance small} and choose symmetries $\Psi_i\in\G_i$ ($i=1,2$) such that
\begin{equation*}
\|\Psi_0\Psi_1F_j'-g_j\|_{p_j}<\delta_0,
\end{equation*}
where $F_j'=e^{-ibt}F_j\circ A^{-1}$. Our goal is to choose $\tilde{\Psi}_j\in\G_i$ such that $\tilde{h}_j:=\tilde{\Psi}_0\tilde{\Psi}_1F_j'-g_j$ satisfies the desired orthogonality conditions.

Define $h_j:=\Psi_0\Psi_1F_j'-g_j$ so that
\begin{align*}
\tilde{h}_j&=\tilde{\Psi}_0\tilde{\Psi}_1F_j'-g_j\\
&=\tilde{\Psi}_0\tilde{\Psi}_1\Psi_1^{-1}\Psi_0^{-1}(h_j+g_j) -g_j\\
&=\Psi_0'\Psi_1'\Psi_0^{-1}(h_j+g_j)-g_j\\
&=\Psi_0'\Psi_1'\Psi_0^{-1}h_j+\Psi_0'\Psi_1'\Psi_0^{-1}g_j-g_j
\end{align*}
by defining $\tilde{\Psi}_1=\Psi_1'\Psi_1$ and $\Psi_0'=\tilde{\Psi}_0$. Recall that the precise forms of some of the symmetries in $\G_1$ depend on the values of the attached parameters. Since $\Psi_1$ originally acts when these parameters are $Id$ and 0, respectively, the form of $\Psi_1$ is fixed as such for the above computation, which merely involves triples of functions and not the attached parameters. For instance, here the translation-modulation symmetry is only used as in \eqref{eq:simplified form}, even if it is composed with other symmetries which would normally alter the attached parameters.

In the above, $\Psi_0$ is given, but its action on functions may be represented in the form

\begin{equation*}
\Psi_0a(x,t)=e^{i\beta t}a(Lx,t)
\end{equation*}
for fixed $\beta\in\R$, $L\in Gl(2d)$, and functions $a:\R^{2d+1}\to\C$.

We have some flexibility in determining $\Psi_0'$, which will be of the form

\begin{equation*}
\Psi_0'a(x,t)=e^{i\gamma t}a(Mx,t)
\end{equation*}
for $\gamma\in\R$ and $M\in Gl(2d)$ to be determined later. (In some sense, we will have $\gamma\sim \beta$ and $M\sim L$.)

Let the action of $\Psi_1'$ be as follows. Given a triple of functions, we first apply the translation symmetry with parameters $u_j=(LU_j,U_j'), w_j=(LW_j,W_j')\in\R^{2d}\times\R$ satisfying $w_1=u_2^{-1}, w_2=u_3^{-1}$, and $w_3=u_1^{-1}$, then apply the sheer symmetry with linear map $\varphi\circ L^{-1}:\R^{2d}\to\R$, dilate the functions by a factor of $r\in\R$, compose them with $S\in Sp(2d)$, modulate by a factor of $e^{i\xi\cdot x}$ with $\xi\in\R^{2d}$, and finally scale by factors of $a_j\in\C$. Thus,

\begin{multline*}
\Psi_1'f_j(x,t)=a_je^{i\xi\cdot x}f_j(LU_j+rSx+LW_j,\\
U_j'+r^2t+\varphi(L^{-1}Sx)+W_j'+\sigma(LU_j,rSx)+\sigma(LU_j,LW_j)+\sigma(rSx,LW_j)).
\end{multline*}

In combining the above symmetries, we have

\begin{multline}\label{eq:hj tilde expansion 1}
\Psi_0'\Psi_1'\Psi_0^{-1}g_j(x,t)-g_j(x,t)=a_je^{i\xi\cdot x-i\beta\varphi(L^{-1}SMx)}e^{i\gamma t-i\beta r^2t-i\beta(U_j'+W_j')}\\
\times e^{-i\beta[\sigma(LU_j,rSMx)+\sigma(LU_j,LW_j)+\sigma(rSMx,LW_j)]}\\
\times g_j(U_j+W_j+rL^{-1}SMx)\\
\times g_j(U_j'+r^2t+\varphi(L^{-1}SMx)+W_j'+\sigma(LU_j,rSMx)+\sigma(LU_j,LW_j)+\sigma(rSMx,LW_j))\\
-g_j(x)g_j(t).
\end{multline}

Here, we use $g_j(y)=e^{-\gamma_j|y|^2}$ for $y$ lying in any of $\R$, $\R^{2d}$, or $\R^{2d+1}$ depending on context.


Similarly,

\begin{multline}\label{eq:hj tilde expansion 2}
\Psi_0'\Psi_1'\Psi_0^{-1}h_j(x,t)=a_je^{i\xi\cdot x-i\beta\varphi(L^{-1}SMx)}e^{i\gamma t-i\beta t-i\beta U_j'-i\beta W_j'}\\
\times e^{-i\beta[\sigma(LU_j,rSMx)+\sigma(LU_j,LW_j)+\sigma(rSMx,LW_j)]}\\
\times h_j(U_j+W_j+rL^{-1}SMx,\\ U_j'+r^2t+\varphi(L^{-1}SMx)+W_j'+\sigma(LU_j,rSMx)+\sigma(LU_j,LW_j)+\sigma(rSMx,LW_j))
\end{multline}

Write $a_j=1+b_j, M=S^{-1}L(Id+K), r=1+s, \gamma=\beta(r^2+U_j'+W_j')+\alpha$, and $\xi\cdot x=\zeta+\beta\varphi((Id+K)x) $. Expanding the factors of \eqref{eq:hj tilde expansion 1}, we obtain for the fifth factor

\begin{multline*}
g_j(U_j+W_j+rL^{-1}SMx)=g_j(U_j+W_j+x+Kx+sx+sKx)\\
=g_j(x)e^{-\gamma_j[|U_j+W_j+x+Kx+sx+sKx|^2-|x|^2]}\\
=g_j(x)+g_j(x)x\cdot (-2\gamma_j)(U_j+W_j+Kx+sx)+O((|U_j|+|W_j|+\|K\|+|s|)^2),
\end{multline*}
where $O((|U_j|+|W_j|+\|K\|+|s|)^2)$ represents the $L^{p_j}$ norm of the remainder term.

Applying similar methods to the sixth factor,

\begin{multline*}
g_j(U_j'+r^2t+\varphi(L^{-1}SMx)+W_j'+\sigma(LU_j,rSMx)+\sigma(LU_j,LW_j)+\sigma(rSMx,LW_j))\\
=g_j(t)+g_j(t)t[-2s-2\gamma_j(U_j'+W_j'+\varphi(x))]+O((\|L^TJL\|+|U_j'|+|W_j'|+\|K\|+|s|+\|\varphi\|)^2)
\end{multline*}

Since the fifth and sixth factors together include a factor of $g_j(x,t)$, we are interested in terms which are not insignificant when multiplied by this $g_j$. For the first three factors,

\begin{multline*}
a_je^{i\xi\cdot x-i\beta\varphi(L^{-1}SMx)}e^{i\alpha t}=(1+b_j)(1+i\zeta\cdot x+O(|\zeta|^2))(1+i\alpha t+O(|\alpha|^2))\\
=1+b_j+i\zeta\cdot x+i\alpha t+O((|\zeta|+|\alpha|)^2),
\end{multline*}
where $O((|\zeta|+|\alpha|)^2)$ represents the $L^{p_j}$ norm of the remainder term when multiplied by the fifth and sixth factors. Similarly, for the fourth factor

\begin{multline*}
e^{-i\beta[\sigma(LU_j,rSMx)+\sigma(LU_j,LW_j)+\sigma(rSMx,LW_j)]}\\
=1-i\beta[r(Id+K)^TL^TJ^TLU_j\cdot x+L^TJLU_j\cdot W_j+r(Id+K)^TL^TJLW_j\cdot x]\\
+O((\beta\|L^TJL\|+\|L^TJL\|+|U_j|+|W_j|+\|K\|+|s|)^2)\\
=1+O((\beta\|L^TJL\|+\|L^TJL\|+|U_j|+|W_j|+\|K\|+|s|)^2).
\end{multline*}

Combining the above factors gives the following expression for \eqref{eq:hj tilde expansion 1}
\begin{multline}\label{eq:for implicit}
g_j(x,t)(b_j+[-2\gamma_j(U_j+W_j+Kx+sx)+i\zeta]\cdot x\\
+[-2st-2\gamma_j(U_j'+W_j'+\varphi(x))+i\alpha]t)+\text{2nd order terms}.
\end{multline}

Recall that $\|L^TJL\|,\beta\|L^TJL\|=O(\delta)$; thus, terms such as $\|L^TJL\|^2$ or $\beta\|L^TJL\|\cdot\|K\|$ are considered 2nd order in a manner that will be made precise shortly.

We now test our expression for $\tilde{h}_j$ in the inner product with the $P_\alpha^{\tau_j}g_j^{p_j-1}$. One may check, using the expression from \eqref{eq:hj tilde expansion 2}, that

\begin{equation*}
\langle\Psi_0'\Psi_1'\Psi_0^{-1}h_j,P_\alpha^{\tau_j}g_j^{p_j-1}\rangle=\langle h_j,P_\alpha^{\tau_j}g_j^{p_j-1}\rangle+(\text{1st order terms})\|h_j\|_{p_j}.
\end{equation*}

The proof will conclude upon applying the Implicit Function theorem to the map $(b_j,\zeta,\alpha,U_j,U_j',K,s,\varphi)\mapsto(\langle\tilde{h}_j,P_\alpha^{\tau_j}g_j^{p_j-1}\rangle:j\in J)$, where $J$ is the collection of indices mentioned in Theorem \ref{thm:Christ intermediate}. This will guarantee a small neighborhood of $(\langle h_j,P_\alpha^{\tau_j}g_j^{p_j-1}\rangle:j\in J)$ in which $(\langle\tilde{h}_j,P_\alpha^{\tau_j}g_j^{p_j-1}\rangle:j\in J)$ may obtain any value for some set of parameters $b_j$,...,etc. The quantitative bounds in the $O(\cdot)$ expressions guarantee that for small enough $\delta_0$, the origin is included in this neighborhood. The uniformity in these bounds guarantees the same $\delta_0$ works in all cases.

To verify the hypotheses of the Implicit Function theorem, we must show that the map
\begin{multline*}
(b_j,\zeta,\alpha,U_j,U_j',K,s,\varphi)\mapsto\\
\langle \big(b_j+[-2\gamma_j(U_j+W_j+Kx+sx)+i\zeta]\cdot x\\
+[-2\gamma_j(U_j'+W_j'+\varphi(x))+i\alpha-2st]t\big)g_j,P_\alpha^{\tau_j}g_j^{p_j-1}\rangle
\end{multline*}
is surjective.

One may see by inspection that since $b_j\in\C$, the terms containing $b_j$ correspond perfectly to the $|\alpha|=0$ cases in the orthogonality conditions. The conditions on $u_j$ and $w_j$ (and therefore on $(U_j,U_j')$ and $(W_j,W_j')$) give enough freedom to determine the $|\alpha|=1$ conditions of just the real parts for $j=1,2$. Together, $\xi$ and $\alpha$ give the condition for imaginary parts when $|\alpha|=1$ and $j=3$. Lastly, $K, s$, and $\varphi$ are together in bijective correspondence with the set of symmetric matrices on $\R^{2d}\times\R$; hence, they give the condition for $|\alpha|=2, j=3$.

\end{proof}

\section{Putting it all Together}\label{sec:main proof}

\begin{proof}[Proof of Theorem \ref{thm: main stability}]
Recall that by the translation scheme of Section \ref{sec: Reduce}, it suffices to prove Theorem \ref{thm: main distance}. By Theorem \ref{thm: Christ qualitative}, it suffices to prove Theorem \ref{thm: main distance} under the assumption of small perturbations.

Let $h_j\in L^{p_j}(\R^{2d+1})$, $B$ be a $(2d)\times(2d)$ matrix, and $r\in\R$ such that 
\begin{equation*}
\dist_\bp(\tilde{\scriptO}(\bh,B,r),(\bg,0,0))<\delta_0.
\end{equation*}
By the Balancing Lemma, there exists $F_j\in L^{p_j}$, a $(2d)\times(2d)$ matrix $A$, and $b\in\R$ such that $(\bF,A,b)\in\tilde{\scriptO}(\bh,B,r)$ and the orthogonality conditions of Theorem \ref{thm:Christ intermediate} hold for $\bF$. Choose $\eta>0$ such that the conclusion of Theorem \ref{thm:Christ intermediate} holds. Define $f_j=F_j-g_j$.

Since 
\begin{equation*}
\dist_\bp(\tilde{\scriptO}(\bh,B,r),(\bg,0,0))^2\leq \|\bf\|_\bp^2+\|A^TJA\|^2+b^2\|A^TJA\|^2,
\end{equation*}
it suffices to show there exists $c>0$ depending only on $d$ and $\bp$ such that

\begin{equation}\label{eq:last line}
\frac{\T(\bg+\bf,A,b)}{\prod_j\|g_j+f_j\|_{p_j}}\leq \bA_\bp^{2d+1}-c(\|\bf\|_\bp^2+(1+b^2)\|A^TJA\|^2).
\end{equation}

By the expansion
\begin{multline*}
\frac{\T(\bg+\bf,A,b)}{\prod_j\|g_j+f_j\|_{p_j}}=\frac{\T(\bg+\bf,A,b)-\T(\bg+\bf,A,0)}{\prod_j\|g_j+f_j\|_{p_j}}\\+\frac{\T(\bg+\bf,A,0)-\T(\bg+\bf,0,0)}{\prod_j\|g_j+f_j\|_{p_j}}+\frac{\T(\bg+\bf,0,0)}{\prod_j\|g_j+f_j\|_{p_j}},
\end{multline*}
Propositions \ref{prop:expand in group structure} and \ref{prop:expansion in twisting}, and Theorem \ref{thm:Christ intermediate}, we have

\begin{multline*}
\frac{\T(\bg+\bf,A,b)}{\prod_j\|g_j+f_j\|_{p_j}}\leq \bA_\bp^{2d+1}-c(\|\bf\|_\bp^2+(1+b^2)\|A^TJA\|^2)\\
+O(\|\bf_\sharp\|_\bp\|\bf_\flat\|_\bp+\|\bf_\flat\|_\bp^2)-\tilde{c}\sum_j\|f_{j,\flat}\|_{p_j}^{p_j}.
\end{multline*}

In the above, some sacrifice is made in the constant $c$ to absorb the $o(\cdot)$ terms. Since $p_j<2$ for all $j$ and we are working under the assumption of small perturbations, one may absorb the $O(\|\bf_\flat\|_\bp^2)$ term, again at the small expense of constants.

If $\sum_j\|f_{j,\flat}\|_{p_j}^{p_j}$ is small relative to $\|\bf\|_\bp^2$, then the $O(\|\bf_\sharp\|_\bp\|\bf_\flat\|_\bp$ term is negligible, as each $\|f_{j,\flat}\|_{p_j}$ is small. (Specifically, one may split into cases where $\|f_{j,\flat}\|_{p_j}\geq\|f_j\|_{p_j}^{(4-p_j)/2}$ for at least one $j$ or none of the $j$.) However, if $\sum_j\|f_{j,\flat}\|_{p_j}^{p_j}$ is large relative to $\|\bf\|_\bp^2$, then the last term dominates (as $p_j<2$), and the above is still negligible.

Thus, we are left with \eqref{eq:last line}, completing the proof.
\end{proof}

\section{Variant Theorem Statements}\label{sec: variants}

While the main result of \cite{ChristSY} is Theorem \ref{thm:ChristSY}, the paper also addresses cases when the hypothesis that $p_j<2$ for $1\leq j\leq 3$ is violated. We are able to extend two of these results-- one positive, the other negative-- to the case of the Heisenberg group without significant further effort.

\begin{theorem}\label{thm:variant 1}
Let $d\geq1$ and $\bp\in(1,2]^3$ be admissible. Then, there exists a $c>0$ with the following property. Let $\bf\in L^\bp(\H^d)$ and $\|f_j\|_{p_j}\neq0$ for each $j\in\{1,2,3\}$. Let $\delta\in(0,1)$ and suppose that $|\T_{\H^d}(\bf)|\geq(1-\delta)\bA_\bp^{2d+1}\prod_j\|f_j\|_{p_j}$. Then there exists a $\bp$-compatible $c\sqrt{\delta}$-diffuse ordered triple of Gaussians $\bG=(G_1,G_2,G_3)$ such that
\begin{equation}\label{eq:variant 1 conclusion}
\|f_j-G_j\|_{p_j}<c\sqrt{\delta}\|f_j\|_{p_j}\text{    for }j\in\{1,2,3\}.
\end{equation}
\end{theorem}

Note Theorem \ref{thm:variant 1} is exactly the same as Theorem \ref{thm: main stability}, except that the case where one of the $p_j=2$ is included at the expense of uniformity in $c$.

\begin{proof}
Without loss of generality, assume $p_1=2$. Since $p_2,p_3>1$ and $\sum_{j=1}^3p_j^{-1}=2$, $p_2,p_3<2$. The proof of Theorem \ref{thm:variant 1} mimics that of Theorem \ref{thm: main stability}, except one takes $f_{1,\sharp}=f_1$, hence $f_{1,\flat}=0$. We must now check every step where the particular properties of $f_{1,\sharp}$ and $f_{1,\flat}$ are used.

In treating the $T'(g_1,f_{2,\sharp},f_{3,\sharp})$ term, the earlier proof of Lemma \ref{lemma: T' 1g 2f} suffices as $f_1$ is not present. However, the proof must be modified when treating terms like $T'(f_{1,\sharp},f_{2,\sharp},g_3)$. Only split $f_{2,\sharp}=f_{2,\sharp,>M_2}+f_{2,\sharp,\le M_2}$, expanding

\begin{equation*}
T'(f_1,f_{2,\sharp},g_3)=T'(f_1,f_{2,\sharp,>M_2},g_3)+T'(f_1,f_{2,\sharp,\le M_2},g_3).
\end{equation*}

The previous analysis using the trivial bound suffices to control $T'(f_1,f_{2,\sharp,>M_2},g_3)$ as $\|f_{2,\sharp,>M_2}\|_{p_2}=\|f_{2,\sharp}\|_{p_2}^2$ still holds. For the remaining term, split $f_1=f_{1,>2M_2}+f_{1,\le2M_2}$, where $f_{1,>2M_2}=f_1\one_{B(0,2M_2)^c}$ and $f_{1,\le2M_2}=f_1\one_{B(0,2M_2)}$.

For $T'(f_{1,\le 2M_2},f_{2,\sharp,\le M_2},g_3)$, again the previous analysis will do; this time we simply gain extra powers of $\log(\|f_{2,\sharp}\|_{p_2})$ rather than $\log(\|f_{1,\sharp}\|_{p_1})$.

For $T'(f_{1,> 2M_2},f_{2,\sharp,\le M_2},g_3)$, we obtain the following integral:

\begin{multline*}
\iint_{B(0,2M_2)^c\times B(0,M_2)}f_1(z_1)f_{2,\sharp}(z_2)g_3(z_1+z_2)\\\times\left[-2\gamma_3\alpha\beta-\gamma_3\beta^2+2\gamma_3^2\alpha^2\beta^2+O(\alpha^3\beta^3)+O(\beta^4)\right]dz_1dz_2.
\end{multline*}

As before, we focus on the term with $2\gamma_3\alpha\beta$, as bounds for other terms follow similarly.

On the above domain, $|x_1+x_2|\ge M_2$, so $g_3(z_1+z_2)\leq e^{-\gamma_3(|x_1|-M_2)^2-\gamma_3(t_1+t_2)^2}$. Therefore, the term in question is controlled by

\begin{multline*}
2\gamma_3\iint_{B(0,2M_2)^c\times B(0,M_2)}f_1(z_1)f_{2,\sharp}(z_2)|x_1|e^{-\gamma_3(|x_1|-M_2)^2}|x_2|\\
\times\|A^TJA\|\cdot|t_1+t_2|e^{-\gamma_3(t_1+t_2)^2}dz_1dz_2\\
\le C\|f_1\|_{p_1}\|f_2\|_{p_2}M_2\|A^TJA\|=o(\|\bf\|_\bp^2).
\end{multline*}

The proof for Lemma \ref{lemma: T'' 1g 2f} is similar.

Lastly, we observe that an equivalent of Theorem \ref{thm:Christ intermediate} still holds. In \cite{ChristSY}, an equivalent of Theorem \ref{thm:variant 1} holds for convolution on $\R^d$. From the proof of that variant, one may extract a similar intermediate version of Theorem \ref{thm:ChristSY}; however, the $-\tilde{c}\sum_{j=1}^3\|f_{j,\flat}\|_{p_j}^{p_j}$ term is replaced with $-\tilde{c}\sum_{j=2}^3\|f_{j,\flat}\|_{p_j}^{p_j}$. This poses no problems, as $f_{1,\flat}=0$; thus there are no unfavorable terms containing $\|f_{1,\flat}\|_{p_1}$ which need be canceled.
\end{proof}

The following proposition is found in \cite{ChristSY}.

\begin{proposition}\label{prop:Christ pk>2}
Let $\bp\in(1,\infty)^3$ be admissible and let $d\ge1$. Suppose that $p_k>2$ for some index $1\le k\le 3$. Then, there exists no $c>0$ for which
\begin{equation*}
|\T_0(\bf)|\leq \left(\bA^d_\bp-c\dist_\bp(\scriptO_\mathbb{E}(\bf),\bg)^2\right)\prod_j\|f_j\|_{p_j}
\end{equation*}
holds uniformly for all $f_j\in L^{p_j}(\R^d)$.
\end{proposition}

We establish the following variant for our generalized operator.

\begin{proposition}\label{prop:no pj>2}
Let $\bp\in(1,\infty)^3$ be admissible and let $d\ge1$. Suppose that $p_k>2$ for some index $1\le k\le 3$. Then, there exists no $c>0$ for which
\begin{equation*}
|\T(\bf,A,b)|\leq(\bA_\bp^{2d+1}-\dist_\bp(\tilde{\scriptO}(\bf,A,b),(\bg,0,0))^2)\prod_j\|f_j\|_{p_j}
\end{equation*}
holds uniformly for all $f_j\in L^{p_j}(\R^d)$.
\end{proposition}

\begin{proof}
Suppose for the sake of contradiction that such a $c>0$ exists. Then, taking the case $A=0$ and $b=0$, one recovers the statement
\begin{equation*}
|\T_0(\bf)|\leq \left(\bA^{2d+1}_\bp-c\dist_\bp(\scriptO_\mathbb{E}(\bf),\bg)^2\right)\prod_j\|f_j\|_{p_j}
\end{equation*}
for all $f_j\in L^{p_j}(\R^{2d+1})$ uniformly. This is because the symmetries for convolution on $\R^{2d+1}$ contain the symmetries for the generalized operator $\T(\bf,A,b)$ when $A=0$ and $b=0$. (More symmetries means smaller distance and therefore, weaker statement.) This contradicts Proposition \ref{prop:Christ pk>2}. Therefore, no such $c>0$ exists.
\end{proof}

As a result of Proposition \ref{prop:no pj>2} and the translation scheme developed in Section \ref{sec: Reduce}, Theorem \ref{thm: main stability} does not hold in the case $p_k>2$ for some $k$.

\end{document}